\documentclass[9pt,a4paper, leqno]{amsart}

\usepackage{amssymb,amsmath,amsfonts,amsthm,a4}
\usepackage{graphicx, amsbsy, mathrsfs, esint, bbm, dsfont, bm}
\usepackage{color}
\usepackage{mathtools}
\usepackage{eucal}
\usepackage{slashed}

\usepackage{cite}
\usepackage{enumerate}

\usepackage[titletoc,title]{appendix}

\usepackage{mathrsfs} 

\usepackage{graphicx}   
\usepackage{subfigure}  
\usepackage{multirow}
\usepackage{slashed} 

\usepackage{hyperref}
\newtheorem{theorem}{Theorem}
\newtheorem{lemma}{Lemma}[section]
\newtheorem{definition}{Definition}

\newtheorem{proposition}{Proposition}[section]
\newtheorem{assumption}{Assumption}

\newtheorem*{assumption*}{Assumption}

\newtheorem*{remarks*}{Remarks}
\newtheorem*{remark*}{Remark}

\newcommand{\R}{\mathbb{R}} 
\newcommand{\C}{\mathbb{C}} 
\newcommand{\N}{\mathbb{N}} 
\newcommand{\Z}{\mathbb{Z}} 
\newcommand{\eps}{\varepsilon} 

\newcommand{\comment}[1]{}
\numberwithin{equation}{section}

\newcommand{\be}{\begin{equation}}
\newcommand{\ee}{\end{equation}}
\newcommand{\bes}{\begin{equation*}}
\newcommand{\ees}{\end{equation*}}

\newcommand{\ii}{\mathrm{i}}

\renewcommand{\leq}{\leqslant}
\renewcommand{\geq}{\geqslant}

\catcode`@=11
\def\section{\@startsection{section}{1}%
  \z@{1.5\linespacing\@plus\linespacing}{.5\linespacing}%
  {\normalfont\bfseries\large\centering}}
\catcode`@=12

\begin{document}

\title[Symmetry and Uniqueness of Ground States for elliptic PDEs]{On symmetry and uniqueness of ground states for linear and nonlinear elliptic PDEs}

\begin{abstract}
We study ground state solutions for linear and nonlinear elliptic PDEs in $\R^n$ with (pseudo-)differential operators of arbitrary order. We prove a general symmetry result in the nonlinear case as well as a uniqueness result for ground states in the linear case. In particular, we can deal with problems (e.\,g.~higher order PDEs) that cannot be tackled by usual methods such as maximum principles, moving planes, or Polya--Szeg\"o inequalities. Instead, we use arguments based on the Fourier transform and we apply a rigidity result for the Hardy-Littlewood majorant problem in $\R^n$ recently obtained by the last two authors of the present paper.
\end{abstract}

\author[L. Bugiera]{Lars Bugiera}
\address{Institute of Medical Biometry and Statistics, University of Freiburg, Hebelstrasse 11, D-79104 Freiburg i. Br., Germany.}
\email{bugiera@imbi.uni-freiburg.de}

\author[E. Lenzmann]{Enno Lenzmann}
\address{University of Basel, Department of Mathematics and Computer Science, Spiegelgasse 1, CH-4051 Basel, Switzerland.}
\email{enno.lenzmann@unibas.ch}

\author[J. Sok]{J\'er\'emy Sok}
\address{}
\email{jeremyvithya.sok@unipd.it}

\maketitle


\section{Introduction and Main Results}

In this short paper, we study symmetry properties and uniqueness of ground states for linear and nonlinear elliptic PDEs posed on $\R^n$. In particular, we will be interested in a general class of problems (including higher-order PDEs) which cannot be studied by classical methods such as maximum principles or Polya-Szeg\"o inequalities. Instead our approach here is based on Fourier methods together with a classification of the Hardy-Littlewood majorant problem in $\R^n$, which was recently obtained in \cite{LeSo-18}.

For a convenient organization of this paper, we will present our results on linear and nonlinear problems in two separate subsections as follows.

\subsection{Linear Results} 
Let $s > 0$ be a real number. We consider ground states $\psi \in H^s(\R^n)$ of linear equations of the from 
\be \label{eq:lin}
P(D) \psi + V \psi = E \psi  ,
\ee
where $E \in \R$ is the eigenvalue and $V : \R^n \to \R$ denotes a given potential. Here $P(D)$ stands for a self-adjoint, elliptic constant coefficient pseudo-differential operator of order $2s$. More precisely, we assume the following condition.

\begin{assumption} \label{ass:lin}
Let $s>0$. The pseudo-differential operator $P(D)$ is given by
$$
\widehat{(P(D) f)}(\xi) = p(\xi) \widehat{f}(\xi),
$$
with some continuous function $p : \R^n \to \R$ that satisfies the estimates
$$
A |\xi|^{2s} + c \leq p(\xi) \leq B |\xi|^{2s} \quad \mbox{for all $\xi \in \R^n$}
$$
with suitable constants $A >0$, $B> 0$, and $c \in \R$.
\end{assumption}

Let us now suppose that $P(D)$ satisfies Assumption \ref{ass:lin}. We assume that $V : \R^n \to \R$ is a bounded potential\footnote{We could relax this condition to unbounded potentials $V \in L^\infty(\R^n) + L^p(\R^n)$ with $p > \max \{n/2s,1\}$. For the sake of simplicity, we omit this generalization here.}. Hence we can consider the well-defined minimization problem
\be \label{eq:E0}
E_0 = \inf \{ \langle f, (P(D) + V) f \rangle : f \in H^s(\R^n), \; \| f\|_{L^2}=1 \}  > -\infty.
\ee  
Furthermore, if we assume that $V(x) \to 0$ as $|x| \to \infty$ in the sense that  $\{ |V(x)| > \eps \}$ has  finite Lebesgue measure for every $\eps > 0$, it is easy to see that
\be
E_0 \leq \inf_{\xi \in \R^n} p(\xi) = \inf \sigma_{\mathrm{ess}}(H),
\ee
where $\sigma_{\mathrm{ess}}(H)$ denotes the essential spectrum of the self-adjoint operator $H=P(D)+V$ defined via the quadratic form appearing in \eqref{eq:E0}. Provided a minimizer $\psi \in H^s(\R^n)$ for \eqref{eq:E0} exists, it is easy to see $\psi$ solves \eqref{eq:lin} with $E=E_0$. Conversely, any solution $\psi \in H^s(\R^n) \setminus \{0\}$ of \eqref{eq:lin} with $E=E_0$ is a minimizer of problem \eqref{eq:E0} up to a trivial rescaling to ensure the normalization condition $\| \psi \|_{L^2}=1$. Following usual nomenclature in spectral theory of Schr\"odinger operators, we refer to such minimizing solutions $\psi \in H^s(\R^n)$ as \textbf{ground states} for the linear problem \eqref{eq:lin}. To have a better contradistinction for the nonlinear problems discussed below, we will also use the term {\bf linear ground state} sometimes.

In the setting of Schr\"odinger operators when $P(D) = -\Delta$, we remark that uniqueness of ground states $\psi$ (up to a trivial multiplicative constant) is a classical result, which can be proven by an wide array of known methods such as  maximum principles, Polya-Szeg\"o principle, and Perron-Frobenius arguments involving the corresponding heat kernel $e^{t \Delta}$. Also, the fractional case for $P=(-\Delta)^s$ with $0 < s <1$ can be readily tackled with such methods. 

However, it is fair to say that the study of uniqueness of ground states of linear problems like \eqref{eq:lin} becomes quite elusive in the case of operators $P(D)$ with higher order $2s > 1$. In fact, uniqueness of ground states may fail in such cases. But in certain natural cases of interest (e.\,g.~arising from linearizations around ground states of nonlinear PDEs), the potential $V$ does have the noteworthy property of having a negative Fourier transform $\widehat{V} < 0$ almost everywhere. As our first main result in this paper, we prove that ground states for \eqref{eq:lin} are in fact unique (up to a trivial constant) under this condition on $V$.

\begin{theorem}[Uniquenes of Linear Ground States] \label{thm:lin}
Let $n \geq 1$, $s>0$, and suppose that $P(D)$ satisfies Assumption \ref{ass:lin}. Assume that $V : \R^n \to \R$ has a Fourier transform $\widehat{V} \in L^1(\R^n) \cap L^2(\R^n)$ with $\widehat{V}(\xi) < 0$ for almost every $\xi \in \R^n$. Finally, we suppose that $E_0 < \inf_{\xi \in \R^n} p(\xi)$ holds in \eqref{eq:E0}. Then we have the following properties.

\begin{enumerate}
\item[(a)] \textbf{Uniqueness:} The ground state solution $\psi \in H^s(\R^n)$ for \eqref{eq:lin} is unique (up to a constant phase). Moreover, we have the strict positivity property of its Fourier transform 
$$
e^{\ii \theta} \widehat{\psi}(\xi) > 0 \quad \mbox{for all $\xi \in \R^n$},
$$ 
where $\theta \in \R$ is a constant.
\item[(b)] \textbf{Symmetry:} Up to a constant phase, the ground state $\psi$ in (a)  satisfies
$$
\psi(-x) = \overline{\psi(x)} \quad \mbox{for a.\,e.~$x \in \R^n$.}
$$
If, in addition, the symbol $p(-\xi)=p(\xi)$ is even, then $\psi : \R^n \to \R$ is real-valued (up to a constant phase).
\end{enumerate}
\end{theorem}

\begin{remarks*}
{\em
1) Under some technical assumptions, we could also treat the non-generic case when $E_0 = \inf_{\xi \in \R^n}p(\xi) = \inf \sigma_{\mathrm{ess}}(H)$ coincides with the bottom of the essential spectrum of $H=P(D)+V$. However, we omit this discussion here.

2) Note that $V \in L^\infty$ by our assumption that $\widehat{V} \in L^1(\R^n)$. As mentioned above, we could relax our conditions to unbounded potentials $V$. But again in order to keep our focus on its simple main argument, we refrain from considering more general cases here.

3) In some sense, the result above yields a Perron-Frobenius type result (i.\,e. positivity and uniqueness of ground states) but when viewed in Fourier space. Of course, the ground state $\psi(x)$ may fail to be real-valued at all (let alone strictly positive) in $x$-space. In fact, a simple example arises in the linearized problem for traveling solitary waves for dispersion-generalized NLS, e.\,g., the linear ground state of $\psi \in H^s(\R^n)$ for equations of the form
$$
((-\Delta)^s + \ii v \cdot \nabla + V) \psi = E \psi
$$
with $s \geq 1/2$ and $v \in \R^n \setminus \{ 0 \}$ (and $|v| < 1$ when $s=1/2$). It is easy to see that any non-trivial solutions $\psi \in H^s(\R^n)$ must be complex-valued due to the presence of the `boost term' $\ii v \cdot \nabla$. However, the result above shows that (under suitable assumptions on $V$), we always have the strict positivity $e^{\ii \theta} \widehat{\psi}(\xi) > 0$ for all $\xi \in \R^n$.

4) If we additionally assume that $\widehat{\psi} \in L^1(\R^n)$ or, more generally, that $\widehat{\psi}$ is a finite positive measure on $\R^n$, then $\psi : \R^n \to \C$ is a positive definite function in the sense of Bochner. See also below.

5) Notice since $\widehat{V}$ and $V$ are both assumed to be real-valued, the potential $V(-x) = V(x)$ is an even function.
}
\end{remarks*}

\subsection{Nonlinear Results}

We now turn to ground state solutions of nonlinear elliptic PDEs in $\R^n$ with pseudo-differential operators $P(D)$ of arbitrary order. As before, let $s >0$ be a real number.  We consider solutions $Q \in H^s(\R^n)$ of nonlinear elliptic PDEs of the form
\be \label{eq:Q}
P(D) Q + \lambda Q - |Q|^{2\sigma} Q  = 0.
\ee
Here $\sigma > 0$ is a given number, which we later assume to be an integer, and $\lambda \in \R$ denotes a given parameter, which plays the role of a nonlinear eigenvalue. We opted to use the letters $Q$ and $\lambda$ instead of $\psi$ and $E$ above in order to  make a clear distinction between the linear and nonlinear problems considered here.

As before, we suppose that $P(D)$ denotes a pseudo-differential operator with constant coefficients defined in Fourier space as
\be
\widehat{(P(D)u)}(\xi) = p(\xi) \widehat{u}(\xi).
\ee
For the nonlinear problem \eqref{eq:Q}, we now impose the following conditions on $P(D)$, where $S^m_{1,0}$ with $m \in \R$ denotes the usual H\"ormander class of symbols for pseudo-differential operators on $\R^n$.

\begin{assumption} \label{ass:1}
Let $s > 0$ be a real number. We suppose that $P(D)$ is a pseudo-differential operator of order $2s$ having a symbol $p \in S^{2s}_{1,0}$ that satisfies the following conditions.
\begin{enumerate}
\item[$(i)$] \textbf{Real-Valuedness:} The symbol $p : \R^n \to \R$ is real-valued.
\item[$(ii)$]\textbf{Ellipticity Condition:} There exist constants $c > 0$ and $R > 0$ such that
$$
p(\xi) \geq c |\xi|^{2s} \quad  \mbox{for} \quad  |\xi| \geq R.
$$ 
\end{enumerate}
\end{assumption}

For the rest of this subsection, we will always assume that $P(D)$ satisfies Assumption \ref{ass:1}. As a consequence, the operator $P(D)=P(D)^*$ is self-adjoint and bounded below on $L^2(\R^n)$ with operator domain $H^{2s}(\R^n)$. Furthermore, we assume the eigenvalue parameter $\lambda \in \R$ in \eqref{eq:Q} satisfies the condition
\be \label{ineq:lamb}
{-\lambda} < {\inf_{\xi \in \R^n}} p(\xi),
\ee
which is equivalent to saying that $-\lambda$ lies strictly below the essential spectrum $\sigma_{\mathrm{ess}}(P(D))$ (in analogy to the condition on $E_0$ in Theorem \ref{thm:lin} above). As a direct consequence, we obtain the norm equivalence
$$
\langle f, (P(D) + \lambda) f \rangle \simeq \| f \|_{H^{s}}^2 ,
$$
where $\langle f, g \rangle = \int_{\R^n} \overline{f} g$ denotes the standard scalar product on $L^2(\R^n)$. Likewise, we introduce the critical exponent $\sigma_*(n,s)$ (which is not necessarily an integer) given by
$$
\sigma_*(n,s) = \begin{dcases*} \frac{2s}{n-2s} & for $s < \frac{n}{2}$, \\ +\infty & for $s \geq \frac{n}{2}$. \end{dcases*}
$$
Thus exponents $\sigma < \sigma_*(n,s)$ correspond to the $H^s$-{\em subcritical case}, which is the situation we shall consider in this paper\footnote{To avoid technicalities, we shall omit the discussion of the critical case $\sigma = \sigma_*(n,s)$ in this paper.}. Note that we have the Sobolev-type inequality
\be \label{ineq:sob}
\| f \|_{L^{2 \sigma+2}}^2 \leq C \langle f, (P(D) + \lambda) f \rangle
\ee
for any $f \in H^s(\R^n)$, where $C> 0$ denotes a suitable constant. Due to the subcriticality $\sigma < \sigma_*(n,s)$, standard variational methods yield existence of an optimal constant $C> 0$ as well as the existence of optimizers $Q \in H^s(\R^n)$ for \eqref{ineq:sob}, which are easily seen to solve \eqref{eq:Q} after a suitable rescaling $Q \mapsto \alpha Q$ with some constant $\alpha$. In fact, we relate this  fact to our definition of ground state solutions for \eqref{eq:Q} as follows.

\begin{definition}
With the notation and assumptions above, we say that $Q \in H^s(\R^n) \setminus \{ 0 \}$ is a \textbf{ground state solution}  if $Q$ solves equation \eqref{eq:Q} and optimizes inequality \eqref{ineq:sob}. 
\end{definition}

Equivalently, as shown in Lemma \ref{lem:Qequiv} below, we obtain that $Q \in H^s(\R^n) \setminus \{ 0 \}$ is a ground state solution for \eqref{eq:Q} if  and only if $Q$ minimizes the action functional
\be
\mathcal{A}(f) = \frac 1 2 \langle f, (P(D)+\lambda) f \rangle - \frac{1}{2 \sigma+2} \| f \|_{L^{2 \sigma+2}}^{2 \sigma +2}
\ee
among all  its non-trivial critical points. Thus the set of ground state solutions is given by
\be \label{def:G}
\mathcal{G} = \{ Q \in K : \mbox{$\mathcal{A}(Q) \leq \mathcal{A}(R)$ for all $R \in K$}\},
\ee
where $K= \{ u \in H^s(\R^n) \setminus \{ 0 \} : \mathcal{A}'(u) = 0 \}$.

We now turn to the question of {\em symmetries} for ground states solutions for \eqref{eq:Q}. As consequence of the real-valuedness of the symbol $p(\xi)$, we notice the reflection-conjugation property
\be
(P(D)f)(-x) = \overline{(P(D)f)(x)}.
\ee
Based on this observation, we may ask whether all ground state solutions $Q$ `inherit' this symmetry property by their variational characterization. In fact, we will prove the following result in this paper when the exponent $\sigma \in \N$ is an integer.

\begin{theorem}[Symmetry for Nonlinear Ground States] \label{thm:main}
Let $n \geq 1$, $s > 0$, and $\sigma \in \N$ with $1 \leq \sigma < \sigma_*(s,n)$. Suppose $Q \in H^{s}(\R^n) \setminus \{ 0\}$ is a ground state solution of \eqref{eq:Q} where $\lambda \in \R$ satisfies \eqref{ineq:lamb}. Finally, we assume that $e^{a |\cdot|} Q \in L^2(\R^n)$ for some $a > 0$. Then it holds that 
$$
Q(x) = e^{\ii \alpha} Q^\bullet(x+x_0)
$$
with some constants $\alpha \in \R$ and $x_0 \in \R^n$.  Here $Q^\bullet : \R^d \to \C$ is a smooth, bounded, and positive definite function in the sense of Bochner. As a consequence, it holds that
$$
Q^\bullet(-x) = \overline{Q^\bullet(x)} \quad \mbox{and} \quad Q^\bullet(0) \geq |Q^\bullet(x)| \quad \mbox{for all $x \in \R^n$}.
$$

If, in addition, the operator $P(D)$ has an even symbol $p(\xi) = p(-\xi)$, the function $Q^\bullet$ must be real-valued (up to a trivial constant complex phase). Consequently, any ground state $Q$ for \eqref{eq:Q} is real and even, i.\,e., we have $Q(-x) =Q(x)$ for all $x \in \R^n$.
\end{theorem}

\begin{remarks*}{\em
1) In Theorem \ref{thm:exp} below, we shall give an analyticity condition on $P(D)$ that ensures the exponential decay property $e^{a |\cdot|} Q \in L^2(\R^n)$ for some $a > 0$.  In particular, it applies to operators of the form
$$
P(D) = c_{k} (-\Delta)^{k} + \sum_{\alpha \in \N^n, |\alpha| \leq m/2-1} c_\alpha (-\ii \partial_x)^{\alpha}
$$
with positive $c_k > 0$, $k \geq 1$, and real arbitrary coefficients $c_\alpha \in \R$. For example, we could take $P(D) = \Delta^2 - \mu \Delta$ with any $\mu \in \R$. Another important class is given by the pseudo-differential operators
$$
P(D) = (1-\Delta)^{s} \quad \mbox{for any $s >0$}.
$$

2) The proof of Theorem \ref{thm:main} will be based on the recent characterization \cite{LeSo-18} of the case of equality in \textbf{Hardy-Littlewood majorant problem} in $\R^n$. Here the topological property that the set $\Omega = \{ \xi \in \R^n : |\widehat{Q}(\xi)| > 0 \}$ is connected in $\R^n$ will enter in an essential way.

3) The function $Q^\bullet: \R^n \to \C$ will be obtained by taking the absolute value on the Fourier side, i.\,e., we set $Q^\bullet = \mathcal{F}^{-1} ( |\mathcal{F} Q|)$. See Section \ref{sec:prelim} for more details.

4) If the symbol $p=p(|\xi|)$ is radially symmetric and strictly increasing in $|\xi|$, then we actually can show that $Q = Q^\sharp$ holds (up to tranlation and complex phase), where $Q^\sharp$ denotes the symmetric-decreasing Fourier rearrangement of $Q$. See \cite{LeSo-18}. For symbols $p$ with cylindrical symmetry, we refer to \cite{BuLeScSo-20}.
}
\end{remarks*}

Next, we turn to the question whether (not necessarily ground state) solutions $Q \in H^s(\R^n)$ of \eqref{eq:Q} satisfy the exponential decay estimate that $e^{a |\cdot|} Q \in L^2(\R^n)$ for some $a > 0$, which is a condition imposed in Theorem \ref{thm:main} above. In fact, we can adapt an analytic continuation argument originally developed to study exponential decay of eigenfunctions of Schr\"odinger operators due to Combes and Thomas \cite{CoTh-73}, building upon O'Connors work \cite{OCon-73}. Here is a list of sufficient conditions on $P(D)$ to carry out such an argument in our case.

\begin{assumption} \label{ass:2}
Suppose $P(D)$ has a symbol $p(\xi)$ which has an analytic continuation to the strip $T_\delta = \{ z \in \C^n : |\mathrm{Im} \, z| < \delta \}$ with some $\delta > 0$. Moreover, we assume the following conditions.
\begin{enumerate}
\item[$(i)$] For each $\kappa \in T_\delta$, there exist constant $\gamma \in \R$ and $\theta \in [0, \pi/2)$ such that
$$
\left |\arg (p(\xi+\kappa) - \gamma) \right | \leq \theta \quad \mbox{for all} \quad \xi \in \R^n.
$$
\item[$(ii)$] For each $\kappa \in T_\delta$, there exist constants $a_1, a_2 >0$ and $b_1, b_2 \in \R$ such that
$$
a_1 |\xi|^{2s} - b_1 \leq \mathrm{Re}(p(\xi + \kappa)) \leq a_2 |\xi|^{2s} + b_2 \quad \mbox{for all} \quad \xi \in \R^n.
$$
\end{enumerate}
\end{assumption}

\begin{remark*}
{\em 
It is elementary to check that any polynomial $p(\xi) = \sum_{|\alpha| \leq m} c_\alpha \xi^{\alpha}$ with coefficients $c_\alpha \in \R$ and $\inf_{\xi \in \R^n} p(\xi) > -\infty$ satisfies the above conditions (with $m=2s$). In particular, operators of the form
$$
P(D) = \Delta^2 - \mu \Delta + \ii v \cdot \nabla \quad \mbox{with} \quad \mu \in \R, v \in \R^n
$$
fall under the scope of Assumption \ref{ass:2}. Also, one can verify that the same is true for operators $P(D) = (1-\Delta)^{s}$ with $s> 0$.
}
\end{remark*}

We can now state the following result, which established the assumed exponential decay $e^{a |\cdot|} Q \in L^2(\R^n)$ for some $a > 0$ appearing in Theorem \ref{thm:main} above. 

\begin{theorem}[Exponential Decay] \label{thm:exp}
Let $n,s$, and $\sigma$ be as in Theorem \ref{thm:main}. If $P(D)$ satisfies Assumption \ref{ass:2}, then any solution $Q \in H^s(\R^n)$ of \eqref{eq:Q} satisfies $e^{a |\cdot|} Q \in L^2(\R^n)$ for some $a > 0$. As a consequence, the conclusions of Theorem \ref{thm:main} hold true.
\end{theorem}

\begin{remark*}
{\em For an in-depth analysis of exponential decay of eigenfunctions of $P(D)+V$ with polynomial symbol $p(\xi)$, we refer to the recent work \cite{HeSk-15}. However, for our purposes here, it is sufficient to obtain a `coarse' exponential decay estimate saying that $e^{a |\cdot|} Q \in L^2(\R^n)$ for some $a>0$.}
\end{remark*}

\subsection{Strategy of the Proofs}
Let us briefly describe the strategy behind the proofs of our main results. The idea to prove Theorems \ref{thm:lin} and \ref{thm:main} is based on taking absolute values of the Fourier transform. That is, for a given function $f \in L^2(\R^n)$, we define 
\be
f^\bullet = \mathcal{F}^{-1} ( |\mathcal{F} f|).
\ee
By Plancherel's identity, we immediately find that $\|f^\bullet\|_{L^2} = \| f \|_{L^2}$ and $\langle f^\bullet, P(D) f^\bullet \rangle = \langle f, P(D) f \rangle$. Moreover, for potentials $V : \R^n \to \R$ as in Theorem \ref{thm:lin} as well as for integers $\sigma \in \N$ with $1 \leq \sigma < \sigma_*(s,n)$, we readily obtain the inequalities\footnote{See also the remark following Lemma \ref{lem:element} for the case of non-integer $\sigma$.}
\be
\langle f^\bullet, V f^\bullet \rangle \leq \langle f, V f \rangle \quad \mbox{and} \quad \| f \|_{L^{2 \sigma+2}} \leq \| f^\bullet \|_{L^{2 \sigma +2}}
\ee  
for any $f \in H^s(\R^n)$. Thus if $\psi \in H^s(\R^n)$ and $Q \in H^s(\R^n)$ are ground states for \eqref{eq:lin} and \eqref{eq:Q}, respectively, so are the functions $\psi^\bullet$ and $Q^\bullet$. Therefore, the conclusions of Theorems \ref{thm:lin} and \ref{thm:main} will follow once we can show that 
\be
\widehat{\psi}(\xi) = e^{\ii \theta} |\widehat{\psi}(\xi)| \quad \mbox{and} \quad \widehat{Q}(\xi) = e^{\ii (\alpha + \beta \cdot \xi)} |\widehat{Q}(\xi)|
\ee
with some constants $\theta, \alpha \in \R$ and $\beta \in \R^n$. We remark that $\widehat{\psi}$ and $\widehat{Q}$ are easily seen to be continuous functions in our setting.

 In terms of harmonic analysis, we are faced to solve a {\em phase retrieval problem}, i.\,e., given the modulus of the Fourier transform of a function, we try reconstruct its phase by exploiting some additional facts. For the linear problem \eqref{eq:lin}, this is an elementary task provided that the potential $V$ satisfies the hypothesis of Theorem \ref{thm:lin}. Not surprisingly, the nonlinear problem \eqref{eq:Q} is harder to analyze. Here, a rigidity result for the so-called {\em Hardy--Littlewood majorant problem} in $\R^n$ (recently obtained in \cite{LeSo-18}) enters in an essential way; see also Lemma \ref{lem:HLM} below. In order to apply this result, we must verify the topological property that 
\be
\Omega = \{ \xi \in \R^n : |\widehat{Q}(\xi)| > 0 \}
\ee
is a {\em connected} set in $\R^n$. To prove this fact (where indeed we show that $\Omega = \R^n$ holds in our case), we will make use of analyticity argument: By standard Payler--Wiener arguments, the exponential decay $e^{a |\cdot|} Q \in L^2(\R^n)$ for some $a> 0$ will ensure that $\widehat{Q}(\xi)$ is analytic in some complex strip around $\R^n$. The analyticity of $\widehat{Q}$ together with the fact $Q$ solves \eqref{eq:Q} will then yield the desired result.

Finally, we recall from above that the  proof of Theorem \ref{thm:exp} is based on a strategy for deriving exponential decay for $N$-body Schr\"odinger operators due to Combes and Thomas \cite{CoTh-73} based on O'Connor's lemma \cite{OCon-73}. 

\subsection*{Acknowledgments} The authors are grateful to the Swiss National Science Foundation (SNSF) for financial support under Grant No.~20021-169464. E.\,L.~also thanks the Mittag--Leffler Institute for its kind hospitality during a stay in March 2019, where parts of this work were done. Moreover, the authors are also grateful to Tobias Weth for valuable comments and to J\'er\^ome Hilken for his careful proofreading. Finally, E.\,L.~thanks Rowan Killip for drawing his attention to the work in \cite{Bo-62}.

\section{Preliminaries} \label{sec:prelim}

\subsection{Fourier Inequalities and Hardy-Littlewood Majorant Problem in $\R^n$}

For a function $f \in L^1(\R^n)$, we define its Fourier transform by
\be
(\mathcal{F} f)(\xi) \equiv \widehat{f}(\xi) = \int_{\R^n} f(x) e^{-2 \pi \ii x \cdot \xi} \, d x,
\ee
with the usual extension to $f \in L^2(\R^n)$ by density. For $f  \in L^2(\R^n)$ given, we recall that the function  $f^\bullet \in L^2(\R^n)$ is obtained by taking the absolute value on the Fourier side, i.\,e., we set
\be 
f^\bullet = \mathcal{F}^{-1} (|\mathcal{F} f|).
\ee 
From Plancherel's identity it is clear that $\| f \|_{L^2} = \| f^\bullet \|_{L^2}$ holds. We record some further elementary properties of this operation.

\begin{lemma} \label{lem:element}
Let $n \geq 1$, $s > 0$, and $\sigma \in \N$ with $\sigma < \sigma_*(s,n)$. 

\begin{enumerate}
\item[$(i)$] For any $f \in H^s(\R^n)$, we have
$$
\langle f^\bullet, P(D) f^\bullet \rangle = \langle f, P(D) f \rangle \quad \mbox{and} \quad \|f\|_{L^{2\sigma+2}} \leq \|f^\bullet \|_{L^{2\sigma+2}}.
$$
\item[$(ii)$] For any $f \in L^2(\R^n)$, it holds  that $f^\bullet(-x) = \overline{{f^\bullet(x)}}$ for a.\,e.~$x \in \R^n$.  

\item[$(iii)$] If $f \in L^2(\R^n)$ and $\widehat{f} \in L^1(\R^n)$, then $f^\bullet : \R^n \to \C$ is a continuous and bounded function which is \textbf{positive definite} in the sense that for any points $x_1, \ldots, x_N \in \R^n$ the matrix $[f^\bullet(x_k-x_l)]_{1\leq k,l \leq N}$ is positive semi-definite, i.\,e.,
$$
\sum_{k,l=1}^N f^\bullet(x_k-x_l) \overline{v}_k v_l \geq 0 \quad \mbox{for all $v \in \C^N$}.
$$
In particular, the inequality $f^\bullet(0) \geq |f^\bullet(x)|$ holds for all $x\in \R^n$.
\end{enumerate}
\end{lemma}

\begin{remark*}
{\em The inequality $\| f \|_{L^{2\sigma+2}} \leq \| f^\bullet \|_{L^{2 \sigma+2}}$ for integer $\sigma \in \N$ is a consequence of the so-called {\em upper majorant property (UMP)} for $L^p$-norms with $p \in 2 \N \cup \{ \infty \}$. That is, for such $p$ and $f, g \in \mathcal{F}(L^{p'}(\R^n))$ we have the implication
$$
\mbox{$|\widehat{f}(\xi)| \leq \widehat{g}(\xi)$ for a.\,e.~$\xi \in \R^n$} \quad \Longrightarrow \quad \| f \|_{L^p} \leq \| g \|_{L^p}.
$$ 
On the other hand, it is well-known that (UMP) fails for $L^p$-norms when $p \not \in 2 \N \cup \{ \infty \}$. Indeed, the known counterexamples (see e.\,g.~\cite{Li-60,Bo-62, MoSc-09}) show the failure of (UMP) in the torus case, i.\,e., for $L^p(\mathbb{T})$. But these examples can be easily transferred to the real line case as follows. Suppose $p> 2$ is not an even integer. Then, as shown in \cite{MoSc-09}, there exist trigonometric polynomials $q$ and $Q$ with Fourier coefficients $|\widehat{q}(n)| = \widehat{Q}(n)$ for all $n \in \Z$ satisfying $\| q \|_{L^p(\mathbb{T})} > \| Q \|_{L^p(\mathbb{T})}$. We can lift this example to Fourier transform in $\R$ by considering the Schwartz functions
$$
q_\lambda(x) = \lambda^{\frac{1}{2p}} q(x) e^{-\lambda x^2}, \quad Q_\lambda(x) = \lambda^{\frac{1}{2p}} Q(x) e^{-\lambda x^2}  
$$
with $\lambda > 0$. It is elementary to check that $\| q_\lambda \|_{L^p(\R)} \to \| q \|_{L^p(\mathbb{T})}$ and $\| Q_\lambda \|_{L^p(\R)} \to \| Q \|_{L^p(\mathbb{T})}$ as $\lambda \to 0^+$. Furthermore, we readily check for the Fourier transforms $|\widehat{q}_\lambda(\xi)| \leq \widehat{Q}_\lambda(\xi)$ for all $\xi \in \R^n$. Thus by taking $\lambda > 0$ sufficiently small, we see that (UMP) fails for $L^p(\R)$ with non-even integer $p$.
}
\end{remark*}

\begin{proof}
First, it is evident that $\langle f, P(D) f \rangle = \int_{\R^n} p(\xi) |\widehat{f}(\xi)|^2 \,d \xi = \langle f^\bullet, P(D) f^\bullet \rangle$. Next, let $p = 2\sigma +2$ with $\sigma \in\N$ with $\sigma < \sigma_*(s,n)$. By H\"older's inequality, we note that $f \in H^s(\R^n)$ implies that $f \in \mathcal{F}(L^{p'}(\R^n))$, i.\,e.~we have $\widehat{f} \in L^{p'}(\R^n)$, where $p' = \frac{2\sigma+2}{2\sigma+1}$ denotes the dual exponent of $p=2\sigma+2$. Thus we can apply to conclude
$$
\| f\|_{L^{2\sigma+2}}^{2 \sigma + 2} = (\widehat{f} \ast \widehat{\overline{f}} \ast \ldots \ast \widehat{f} \ast \widehat{\overline{f}})(0)
$$
with $2\sigma+1$ convolutions on the right-hand side. With the use of the autocorrelation function
$$
\Psi_{\widehat{f}}(\xi) = (\widehat{f} \ast \widehat{\overline{f}})(\xi) = (\widehat{f} \ast \overline{\widehat{f}(-\cdot)})(\xi) = \int_{\R^n} \widehat{f}(\xi+\xi') \overline{\widehat{f}(\xi')} \, d\xi',
$$
we can write 
$$
\| f \|_{L^{2\sigma+2}}^{2\sigma+2} = (\Psi_{\widehat{f}} \ast \ldots \ast \Psi_{\widehat{f}})(0),
$$
where the number of convolutions is equal to $\sigma$. Since $|\Psi_{\widehat{f}}|(\xi) \leq \Psi_{|\widehat{f}|}(\xi)$, we deduce
$$
\| f \|_{L^{2\sigma+2}}^{2\sigma+2} \leq (\Psi_{|\widehat{f}|} \ast \ldots \ast \Psi_{|\widehat{f}|})(0) = \| f^\bullet \|_{L^{2\sigma+2}}^{2 \sigma + 2},
$$
which completes the proof of item (i).

The proof of (ii) is a direct consequence of the fact that $\widehat{f^\bullet}= |\widehat{f}|$ is real-valued. Furthermore, item (iii) is a classical fact using that $\widehat{f^\bullet}=|\widehat{f}| \geq 0$ is non-negative and assuming that $\widehat{f^\bullet} \in L^1(\R^n)$ (or more generally $\widehat{f^\bullet}$ is a finite measure on $\R^n$); see, e.\,g., for a discussion of positive-definite functions and Bochner's theorem.
\end{proof}

As a next essential fact we recall from \cite{LeSo-18} the following \textbf{rigidity result}.

\begin{lemma}[Equality in the Hardy-Littlewood Majorant Problem in $\R^n$] \label{lem:HLM}
 Let $n \geq 1$ and $p \in 2 \N \cup \{\infty\}$ with $p> 2$. Suppose that $f, g \in \mathcal{F}(L^{p'}(\R^n))$ with $1/p+1/p'=1$ satisfy the majorant condition
$$
|\widehat{f}(\xi)| \leq \widehat{g}(\xi) \quad \mbox{for a.\,e.~$\xi \in \R^n$.}
$$
In addition, we assume that $\widehat{f}$ is continuous and that $\{ \xi\in \R^n : |\widehat{f}(\xi)| > 0 \}$ is a connected set. Then equality
$$
\| f \|_{L^p} = \| g \|_{L^p}
$$
holds if and only if
$$
\widehat{f}(\xi) = e^{\ii (\alpha + \beta \cdot \xi)} \widehat{g}(\xi) \quad  \mbox{for all $\xi \in \R^n$},
$$
with some constants $\alpha \in \R$ and $\beta \in \R^n$.
\end{lemma}

\begin{remark*}{\em
The connectedness of the set $\Omega \subset \R^n$ is essential. See also \cite{LeSo-18} for a counterexample when $\Omega$ is not connected. However, as we will show below, the set $\Omega = \{ \xi \in \R^n : |\widehat{Q}(\xi)| > 0 \}$ will turn out to be connected (in fact, we show $\Omega= \R^n$ holds) for the ground states $Q$ of \eqref{eq:Q} in the setting considered in this paper.
}
\end{remark*}

\subsection{Smoothness and Exponential Decay of $Q$}

Recall that we always suppose that $P(D)$ satisfies Assumptions \ref{ass:1}.

\begin{proposition} \label{prop:Qsmooth}
Let $n \geq 1$, $s > 0$, and $\sigma \in \N$ with $1 \leq \sigma < \sigma_*(n,s)$. Then any solution $Q \in H^s(\R^n)$ satisfies $Q \in H^\infty(\R^n) = \bigcap_{k \geq 0} H^k(\R^n)$. 
\end{proposition}

\begin{proof}
This follows from Sobolev embeddings and regularity theory for pseudo-differential operators. For the reader's convenience, we give the details. By picking a sufficiently large constant $\mu > 0$, we can assume that $p(\xi) + \mu \gtrsim \langle \xi \rangle^{2s}$ holds. Hence $Q \in H^s(\R^n)$ solves
\be \label{eq:Qregularity}
(P(D) + \mu ) Q = (Q \overline{Q})^{\sigma} Q + (\mu - \lambda) Q .
\ee
Indeed, let us first suppose that $Q \in H^s(\R^n) \cap L^\infty(\R^n)$. Then $(P(D) + \mu)  Q = (Q \overline{Q})^\sigma Q + (\mu-\lambda) Q \in H^{s} \cap L^\infty(\R^n)$ holds, since $\sigma$ is an integer and $H^s(\R^n) \cap L^\infty(\R^n)$ forms an algebra. Now since $p(\xi) + \mu \gtrsim \langle \xi \rangle^{2s}$, we have that $(P(D)+ \mu)^{-1}$ belongs to class $S^{-2s}_{1,0}$. Therefore $(P(D)+\mu)^{-1} : H^{m}(\R^n) \to H^{m+2s}(\R^n)$ for any $m \in \R$ and we deduce that $Q \in H^\infty(\R^n) = \cap_{k \geq 0} H^k(\R^n)$ by iterating the equation \eqref{eq:Qregularity}.

It remains to show that $Q \in L^\infty(\R^n)$ follows from our assumptions. If $s > n/2$, this is clearly true by Sobolev embeddings. For $0 < s \leq n/2$, we need to bootstrap the equation by using the mapping properties of the inverse $(P(D)+\mu)^{-1}$. Indeed, we note that $|Q|^{2 \sigma} Q \in L^{\frac{p_*}{2\sigma+1}}(\R^n)$ with $p_* = 2n/(n-2s)$ by the Sobolev embedding $H^{s}(\R^n) \subset L^{p_*}(\R^n)$. Since $(P(D)+\mu)^{-1} : H^{m, p}(\R^n) \to H^{m+2s, p}(\R^n)$ for any $m \in \R$ and $1 < p < \infty$, we deduce that $Q \in H^{2s, \frac{p_*}{2\sigma+1}}(\R^n)$, which is a gain of regularity for $Q$. We can proceed this argument to obtain after finitely many steps that $Q \in H^{m,p}(\R^n)$ with $m > n/p$, which yields that $Q \in L^\infty(\R^n)$ by Sobolev embeddings.
\end{proof}

\subsection{On the Notion of Ground State Solutions}
As remarked in the introduction, we have the following simple fact, where we assume $n,s,\sigma$, and $\lambda$ satisfy the assumptions of Theorem \ref{thm:main}. Recall the definition of the set $\mathcal{G}$ in \eqref{def:G}.

\begin{lemma} \label{lem:Qequiv}
$Q \in H^s(\R^n)$ is a ground state solution of \eqref{eq:Q} if and only if $Q \in \mathcal{G}$.
\end{lemma}

\begin{proof}
Let $Q, R \in H^s(\R^n)$ be two non-trivial solutions of \eqref{eq:Q}. By integrating the equation \eqref{eq:Q} against $\overline{Q}$ and $\overline{R}$, we find 
\be \label{eq:QR}
\langle Q,(P(D) + \lambda) Q \rangle = \| Q \|_{L^{2 \sigma+2}}^{2 \sigma+2}, \quad  \langle R, (P(D)+ \lambda) R \rangle = \| R \|_{L^{2 \sigma+2}}^{2 \sigma+2}.
\ee
As a consequence, we get
$$
\mathcal{A}(Q) = \left ( \frac{1}{2} - \frac{1}{2 \sigma+2} \right ) \| Q \|_{L^{2 \sigma+2}}^{2 \sigma+2}, \quad \mathcal{A}(R) =  \left ( \frac{1}{2} - \frac{1}{2 \sigma+2} \right ) \| R \|_{L^{2 \sigma+2}}^{2 \sigma+2} .
$$
Hence we have the equivalence
$$
\mathcal{A}(Q) \leq \mathcal{A}(R) \quad \Longleftrightarrow \quad \| Q \|_{L^{2 \sigma+2}} \leq \| R \|_{L^{2 \sigma+2}}.
$$
Next, let $C> 0$ denote the optimal constant for \eqref{ineq:sob}. From \eqref{eq:QR} we obtain the bounds
$$
\| Q \|_{L^{2 \sigma+2}}^{2 \sigma} \geq \frac{1}{C}, \quad \| R \|_{L^{2\sigma+2}}^{2 \sigma} \geq \frac{1}{C},
$$
where equality occurs if and only if $Q$ and $R$ are optimizers for \eqref{ineq:sob}, respectively.  

Suppose now that $Q$ is a ground state solution, which means an optimizer for \eqref{ineq:sob} by definition. Then we must have $\| R \|_{L^{2 \sigma+2}} \geq \| Q \|_{L^{2\sigma+2}}$. This show that $Q \in \mathcal{G}$.

On the other hand, let us assume that $Q \in \mathcal{G}$. To show that $Q$ must optimize \eqref{ineq:sob}, we argue by contradiction as follows. Suppose $Q$ is not an optimizer. Then $\| Q \|_{L^{2 \sigma+2}} > C^{-1}$. But by taking $R$ to be an optimizer, we deduce that $C^{-1}=\| R \|_{L^{2 \sigma+2}} < \| Q \|_{L^{2 \sigma+2}}$, which contradicts that we must have $\mathcal{A}(Q) \leq \mathcal{A}(R)$.
\end{proof}

\section{Proof of Theorem \ref{thm:lin}}

Let $\psi \in H^s(\R^n)$ be a ground state for \eqref{eq:lin} with $E=E_0 < \inf_{\xi \in \R^n} p(\xi)$. If we set $\lambda = -E$, we can write \eqref{eq:lin} in Fourier space as
\be \label{eq:psi_F}
\widehat{\psi}(\xi) = \frac{1}{p(\xi)+\lambda} ( \widehat{W} \ast \psi )(\xi), \quad \mbox{with $\widehat{W}=-\widehat{V}$.}
\ee  
Note that $\widehat{W} \in L^2(\R^n)$ by assumption and hence $\widehat{(W\psi)} = \widehat{W} \ast \widehat{\psi}$ and, moreover, this is a continuous function because it is the convolution of two $L^2$-functions. Since $p(\xi)+\lambda > 0$ is also continuous by assumption on $p$, we deduce that the Fourier transform $\widehat{\psi}(\xi)$ is a continuous function from \eqref{eq:psi_F}.

Next, we claim that
\be \label{ineq:psi_pos}
|\widehat{\psi}(\xi)| > 0 \quad \mbox{for all $\xi \in \R^n$}.
\ee
To see this, we first note that 
$$
\psi^\bullet = \mathcal{F}^{-1}(|\widehat{\psi}|)
$$ 
is also a ground state solution for \eqref{eq:lin}. Indeed, in view of $\widehat{V}(\xi) < 0$ almost everywhere, we can argue as in the proof of Lemma \ref{lem:element} to conclude
$$
\langle \psi, V \psi \rangle = (\widehat{V} \ast \Psi_{\widehat{\psi}})(0) \geq (\widehat{V} \ast \Psi_{|\widehat{\psi}|})(0) = \langle \psi^\bullet, V \psi^\bullet \rangle,
$$
where we recall that $\Psi_g(\xi) = \int_{\R^n} g(\xi+\eta) \overline{g(\eta)} \, d \eta$ is the autocorrelation function of $g$. Thus from Lemma \ref{lem:element} (i) we readily find that
\begin{align*}
\langle \psi^\bullet, (P(D)+ V) \psi^\bullet \rangle \leq \langle \psi, (P(D)+V) \psi \rangle,
\end{align*} 
whence $\psi^\bullet$ is also a ground state, since we trivially have $\| \psi^\bullet \|_{L^2} = \| \psi \|_{L^2}$.

Therefore, in order to show \eqref{ineq:psi_pos}, we can assume that $\widehat{\psi}(\xi)=|\widehat{\psi}(\xi)| \geq 0$ is non-negative. But from the assumption that $\widehat{W} =-\widehat{V} > 0$ almost everywhere we deduce that $(\widehat{W} \ast \widehat{\psi})(\xi) > 0$ for all $\xi \in \R^n$. By the positivity $p(\xi)+\lambda > 0$, we immediately deduce that \eqref{ineq:psi_pos} holds from \eqref{eq:psi_F}.

Next, we establish the following result.

\begin{proposition} \label{prop:phase_lin}
There exists a constant $\theta \in \R$ such that 
$$
\widehat{\psi}(\xi) = e^{\ii \theta} |\widehat{\psi}(\xi)| \quad \mbox{for all $\xi \in \R^n$.}
$$
\end{proposition}

\begin{proof}[Proof of Proposition \ref{prop:phase_lin}]
By the continuity of $\widehat{\psi}$ and the fact that $|\widehat{\psi}(\xi)| > 0$ for all $\xi \in \R^n$, there exists a continuous function $\vartheta : \R^n \to \R$ such that
\be \label{eq:phase_first}
\widehat{\psi}(\xi) = e^{\ii \vartheta(\xi)} |\widehat{\psi}(\xi)| \quad \mbox{for all $\xi \in \R^n$}.
\ee
Since $\psi$ and $\psi^\bullet$ are both ground states for \eqref{eq:lin}, we must have equality
\be
(\widehat{W} \ast \Psi_{\widehat{\psi}})(0) = (\widehat{W} \ast \Psi_{|\widehat{\psi}|})(0),
\ee
with the autocorrelation function $\Psi_{g}(\xi) = \int_{\R^n} g(\xi+\eta) \overline{g(\eta)} \, d\eta$. In view of \eqref{eq:phase_first}, we conclude
$$
\int_{\R^n \times \R^n} \widehat{W}(\xi) e^{\ii \{ \vartheta(-\xi+\eta)-\vartheta(\eta) \}} |\widehat{\psi}(\xi+\eta)||\widehat{\psi}(\eta)| \, d\xi \, d\eta = \int_{\R^n \times \R^n} \widehat{W}(\xi) |\widehat{\psi}(\xi+\eta)||\widehat{\psi}(\eta)| \, d\xi \, d\eta.
$$
Since $W(\xi)  |\widehat{\psi}(\xi+\eta)||\widehat{\psi}(\eta)| > 0$ for all $(\xi,\eta) \in \R^n \times \R^n$, we deduce that
$$
\vartheta(-\xi+\eta) - \vartheta(\eta) \in 2 \pi \mathbb{Z} \quad \mbox{for all $(\xi, \eta) \in \R^n \times \R^n$.}
$$
By the continuity of $\vartheta$, the difference above must be locally constant. Since $\R^n \times \R^n$ is connected, we infer that 
\be \label{eq:theta}
\vartheta(-\xi+ \eta) - \vartheta(\eta) = c  \quad \mbox{for all $(\xi,\eta) \in \R^n \times \R^n$},
\ee
with some constant $c \in 2 \pi \mathbb{Z}$. But by choosing $\xi=0$, we see that $c=0$ is the only possibility. From the functional equation \eqref{eq:theta} with $c=0$ we readily deduce that $\vartheta(-\xi) = \vartheta(0)$ for all $\xi \in \R^n$. Hence $\vartheta$ is a constant function and by taking $\theta = \vartheta(0) \in \R$, we complete the proof of Proposition \ref{prop:phase_lin}.
\end{proof}

By applying Proposition \ref{prop:phase_lin}, we complete the proof of Theorem \ref{thm:lin} part (i).

The symmetry property in part (ii) directly follows from the fact that $e^{\ii \theta} \widehat{\psi}(\xi) > 0$ together with the elementary property $f(-x) = \overline{f(x)}$ holds a.\,e. for $f \in L^2(\R^n)$ whenever $\widehat{f}(\xi)$ is real-valued. Finally, let us suppose that $p(-\xi) = p(\xi)$ is even. Then $H=P(D)+ V$ is real operator, i.\,e., we have $\mathrm{Re} \, (H f) = H \mathrm{Re}\, f$. In particular, we thus choose any eigenfunction of $H$ to be real-valued and, in particular, this applies to the ground state $\psi$.

The proof of Theorem \ref{thm:lin} is now complete. \hfill $\qed$

\section{Proof of Theorem \ref{thm:main}}

Let $Q \in H^s(\R^n)$ be a ground state solution as in Theorem \ref{thm:main}. We define the set
\be
\Omega = \{ \xi \in \R^n : |\widehat{Q}(\xi)| > 0 \}.
\ee
This is an open set in $\R$, since the function $|\widehat{Q}|$ is continuous due to analyticity of $\widehat{Q}$ is analytic by our assumption $e^{a |\cdot|} Q \in L^2(\R^n)$ for some $a> 0$ and using standard Paley--Wiener arguments.

\begin{lemma} \label{lem:omega}
It holds that $\Omega = \R^n$.
\end{lemma}

\begin{remark*}{\em 
For non-ground state solutions $Q \in H^{s}(\R^n)$ of \eqref{eq:Q}, we expect that $\widehat{Q}$ vanishes at certain points. In fact, we expect that the set $\{ |\widehat{Q}(\xi)| > 0 \}$ is not connected for non-ground state solutions $Q$. }
\end{remark*}

\begin{proof}
In view of Lemma \ref{lem:element}, we remark that $Q^\bullet \in H^s(\R^n)$ is also a ground state solution for \eqref{eq:Q}. Hence we can assume that $\widehat{Q} = |\widehat{Q}| \geq 0$ is non-negative without loss of generality. Next, by applying the Fourier transform to \eqref{eq:Q} and using that $\sigma \in \N$ is an integer, we get
\be
\widehat{Q}(\xi) = \frac{1}{p(\xi) + \lambda} ( \widehat{Q} \ast \ldots \ast \widehat{Q})(\xi)
\ee
with $k=2 \sigma +1 \in \N$ convolutions appearing on the right-hand side. From this identity and Lemma \ref{lem:convo} and iteration, we deduce that $\Omega \subset \R^n$ must be identical to its \textbf{$k$-fold Minkowski sum}, i.\,e.,
\be \label{eq:mink}
\Omega   =   \bigoplus_{m=1}^k \Omega \equiv \left \{  \xi_1 + \ldots + \xi_k : \mbox{$\xi_m \in \Omega$ for $m=1, \ldots, k$} \right \} .
\ee
For the moment, let us now suppose that
\be \label{eq:0_omega}
0 \in \Omega.
\ee
Since $\Omega$ is open, this implies that $B_r(0) \subset \Omega$ for some $r > 0$. By \eqref{eq:mink}, this implies that
$$
\bigoplus_{m=1}^k B_r(0) \subset \Omega.
$$
On the other hand, we readily see that $B_{2r}(0) \subset B_r(0) \oplus B_r(0) \subset \oplus_{m=1}^k B_r(0)$. Iterating this argument, we conclude that
$$
B_{Nr}(0) \subset \Omega \quad \mbox{for all} \quad  N \in \N,
$$
whence it follows that $\Omega=\R^n$ must hold.

Thus it remains to show that \eqref{eq:0_omega} is true. We argue by contradiction as follows. Suppose that $0 \not \in \Omega$ and define the function $F : \R^n \to \R$ by setting
$$
F(\xi) = \widehat{Q}((k-1) \xi) \widehat{Q}(-\xi)
$$
However, we must have
$$
F(\xi) \equiv 0 .
$$
Indeed, if $F(\xi_*) \neq 0$ for some $\xi_* \in \R^n$ then $(k-1) \xi_* \in \Omega$ and $-\xi_* \in \Omega$. This implies that $0 = (k-1) \xi_* - \sum_{m}^{k-1} \xi_* \in \oplus_{m=1}^k \Omega$ so that $0 \in \Omega$ by \eqref{eq:mink}. Thus $0 \not \in \Omega$ implies that $F(\xi) \equiv 0$ vanishes identically. Since $\widehat{Q}((k-1) \xi) \not \equiv 0$, this yields that the function $\widehat{Q}(-\xi)$ must vanish on some non-empty open set in $\R^n$. By the (real) analyticity of $\widehat{Q}: \R^n \to \R$ this implies $\widehat{Q} \equiv 0$ on $\R^n$. But this is a contradiction. 

Thus we have shown that \eqref{eq:0_omega} holds, which completes the proof.
\end{proof}

With the result of Lemma \ref{lem:omega} at hand, we are ready to finish the proof of Theorem \ref{thm:main}. Indeed, if $Q \in H^s(\R^n)$ is a ground state solution, we must necessarily have the equality
$$
\| Q \|_{L^{2 \sigma+2}} = \| Q^\bullet \|_{L^{2 \sigma+2}}.
$$
But we can apply Lemma \ref{lem:HLM} with $f=Q$ and $g = Q^\bullet$ to conclude that $\widehat{Q} = e^{\ii (\alpha + \beta \cdot \xi)}|\widehat{Q}(\xi)|$ for all $\xi$ with some constants $\alpha \in \R$ and $\beta \in \R^n$. Hence we find
$$
Q(x) = e^{\ii \alpha} Q^\bullet(x+x_0)
$$
with the constant $x_0 = -\frac{1}{2\pi} \beta \in \R^n$. The asserted properties of $Q^\bullet$ now follow from Lemma \ref{lem:element} together with the fact that $\widehat{Q^\bullet} \in L^1(\R^n)$,  since we have $(1+|\xi|)^m \widehat{Q} \in L^2(\R^n)$ for $m > n/2$ by Proposition \ref{prop:Qsmooth}. 

Finally, let us additionally assume that the symbol 
$$
p(-\xi) = p(\xi)
$$ 
is even. In this case, we can adapt a trick from \cite{FrLiSa-16} (see also Lemma \ref{lem:trick}) to show that any ground state $Q \in H^s(\R^n)$ must be real-valued up to a trivial constant complex phase, i.\,e., we claim that
\be \label{eq:phase}
e^{\ii \theta} Q(x) \in \R \quad \mbox{for all} \quad \xi \in \R^n  
\ee 
with some constant $\theta \in \R$. To prove this, we decompose
$$
Q = Q_R + \ii Q_I
$$
into real and imaginary part. If either $Q_R \equiv 0$ or $Q_I \equiv 0$, then there is nothing is left to prove. Hence we assume that both parts are non-trivial. From Lemma \ref{lem:trick} we obtain
\be
\langle Q, (P(D) + \lambda) Q \rangle =  \langle Q_R, (P(D)+ \lambda) Q_R \rangle + \langle Q_I, (P(D) + \lambda) Q_I \rangle =: D_R + D_I,
\ee
\be \label{ineq:Lsig}
\| Q \|_{L^{2\sigma+2}}^2 \leq \|Q_R\|_{L^{2 \sigma+2}}^2 + \| Q_I \|_{L^{2 \sigma+2}}^2 =: N_R + N_I. 
\ee
Now let $C >0$ denote the optimal constant for \eqref{ineq:sob}. Since $Q$ is an optimizer, we deduce
$$
C = \frac{\| Q \|_{L^{2 \sigma+2}}^2}{\langle f, (P(D) + \lambda) f \rangle} \leq \frac{N_R+N_I}{D_R+D_I} \leq \max \left ( \frac{N_R}{D_R}, \frac{N_I}{D_I} \right ) \leq C.
$$
This shows that we must have equality in \eqref{ineq:Lsig}, which by Lemma \ref{lem:trick} and $Q_R \not \equiv 0 \not \equiv Q_I$ implies that there is some constant $\alpha >0$ such that $Q_I^2 = \alpha^2 Q_R^2$. We want to establish $Q_I = \pm  \alpha Q_R$. To do so, we apply Lemma \ref{lem:trick} now to the decomposition
$$
Q= e^{\ii \pi/4} Q_a + \ii e^{\ii \pi/4} Q_b
$$
with real-valued functions $Q_a$ and $Q_b$. In fact, an elementary computation shows that $Q_a = \frac{\ii}{\sqrt{2}}(Q_R + Q_I)$ and $Q_b = \frac{1}{\sqrt{2}}(-Q_R+Q_I)$. We still have $|Q(x)|^2 = Q_a(x)^2 + Q_b(x)^2$ and also $\langle Q, (P(D)+\lambda) Q\rangle = \langle Q_a,(P(D)+\lambda Q_a \rangle+ \langle Q_b, (P(D)+\lambda) Q_b \rangle$ by using that $p(-\xi)=p(\xi)$ is even. Now if $Q_a \equiv 0$, then we are done since $Q_I=-Q_R$ in this case. If $Q_a \not \equiv 0$, we obtain $Q_b^2 = \beta^2 Q_a^2$ with some constant $\beta > 0$. Note that $\beta^2 \neq 1$ because otherwise this would imply $Q_R Q_I \equiv 0$ (which would yield $Q \equiv 0$ from using $Q_I^2 = \alpha^2 Q_R^2$). In summary, we conclude 
$$
Q_I^2 = \alpha^2 Q_R^2 \quad \mbox{and} \quad \frac{1}{2} (1 + \alpha^2) (1-\beta^2) Q_R^2 = (1+\beta^2) Q_R Q_I.
$$
But this implies that $Q_I = \pm \alpha Q_R$, which  proves that \eqref{eq:phase} is true.

The proof of Theorem \ref{thm:main} is now complete. \hfill $\Box$

\section{Proof of Theorem \ref{thm:exp}}

We will adapt an elegant idea due Combes and Thomas \cite{CoTh-73} who proved exponential decay of eigenfunctions for $N$-body Schr\"odinger operators by an analytic continuation argument, which is based on O'Connor's Lemma (see Lemma \ref{lem:oconnor} below) together with standard analytic perturbation theory; see \cite{ReSi-78, Ka-95}.

We define the operator  $H=P(D)+V$ with $V = -|Q|^{2 \sigma}$ acting on $L^2(\R^n)$. Note that $V \in L^\infty(\R^n)$ is bounded by Proposition \ref{prop:Qsmooth}. Hence, by standard theory, the operator $H$ is self-adjoint with operator domain $H^{2s}(\R^n)$. In particular, we see that $Q$ is an $L^2$-eigenfunction of $H$ satisfying
$$
H Q = -\lambda Q.
$$
Since $V(x) \to 0$ as $|x| \to \infty$, we have $\sigma_{\mathrm{ess}}(H) = \sigma_{\mathrm{ess}}(P(D)) = \inf_{\xi \in \R^n} p(\xi)$. By our assumption \eqref{ineq:lamb}, we see that the eigenvalue $-\lambda$ lies strictly below the essential spectrum of $H$.

We shall now implement an analytic continuation argument to show that $e^{a |\cdot|} Q \in L^2(\R^n)$ must hold for some sufficiently small $a > 0$. To do so, we adapt an argument due to Combes and Thomas as follows. For real $\kappa \in \R^n$, we can define the unitary operators 
$$
(U(\kappa) f)(x)=e^{2 \pi \ii \kappa \cdot x} f(x)
$$ 
acting on $L^2(\R^n)$. Likewise, we consider the family of unitarily equivalent operators
$$
H(\kappa) = U(\kappa)H U(\kappa)^{-1}.
$$
We  readily find that
$$
U(\kappa) P(D) U(\kappa)^{-1} = P_\kappa(D), \quad U(\kappa) V U(\kappa)^{-1} = V,
$$
where $P_\kappa(D)$ has the shifted symbol $p(\xi+\kappa)$.

Now, by standard Paley-Wiener theory, we note that if $U(\kappa) Q$ has an analytic continuation for $|\mathrm{Im} \, \kappa| < \delta$ then $e^{a |\cdot|} Q \in L^2(\R^n)$ for all $0 < a < \delta$, which would finish the proof. To see that $U(\kappa) Q$ can be analytically continued if $|\mathrm{Im} \, \kappa| < \delta$ for some $\delta > 0$, we prove that $H(\kappa)$ is an {\em analytic family of type (B)} on the complex strip $T_\delta$.  We use a form argument as follows. For any $\kappa \in T_\delta$, we can define the quadratic form
\be
{\tt q}(\kappa)[f,f] = \int_{\R^n} p(\xi+\kappa) |\widehat{f}(\xi)|^2 \, d \xi + \int_{\R^n} V |f|^2 \, dx \quad \mbox{for} \quad f \in H^s(\R^n).
\ee
We claim that $\{ {\tt q}(\kappa)\}_{\kappa \in T_\delta}$ is an analytic family of quadratic forms of type (b) with form domain $H^s(\R^n)$ (in the nomenclature of \cite{ReSi-78}). That is, we have the following properties.
\begin{enumerate}
\item[(1)] For each $\kappa \in T_\delta$, the form ${\tt q}(\kappa)$ is closed and strictly $m$-sectorial with domain $H^s(\R^n)$.
\item[(2)] For each $f \in H^s(\R^n)$, the function $\kappa \mapsto {\tt q}(\kappa)[f,f]$ is analytic in $\kappa \in T_\delta$.
\end{enumerate}
Indeed, by Assumption \ref{ass:2} item (i), we see that ${\tt q}$ is strictly $m$-sectorial (see \cite{ReSi-78} for the relevant definition). To show that ${\tt q}(\kappa)$ is closed on the domain $H^s(\R^n)$, it suffices to show that its real part $\mathrm{Re} ({\tt q})(\kappa)$ is closed, i.\,e., if $f_n \in H^s(\R^n)$  with $f_n \to f$ in $L^2(\R^n)$ and $\mathrm{Re}({\tt q})(\kappa)[f_n-f_m, f_n -f_m] \to 0$ as $m,n \to \infty$ then $f \in H^s(\R^n)$. But this later claim easily from property (ii) in Assumption \ref{ass:2}. This shows (1) above. Finally, we note that (2) obviously holds by our analyticity assumption on the symbol $p$. From the fact that ${\tt q}(\kappa)$ is an analytic family of form of type (b) it follows that the set of associated operators $\{ H(\kappa) \}_{\kappa \in T_\delta}$ defines an analytic family of operators of type (B).

Now, by standard perturbation theory, any discrete eigenvalue $E(\kappa_0)$ of  $H(\kappa_0)$ moves analytically for $\kappa$ close to $\kappa_0$. But if $\mathrm{Im} (\kappa-\kappa_0)=0$, we have that $E(\kappa) = E(\kappa_0)$ since the operators $H(\kappa)$ and $H(\kappa_0)$ are unitarily equivalent in this case. Hence $E(\kappa)$ is constant and remains an eigenvalue as long as it stays away from $\sigma_{\mathrm{ess}}(H(\kappa))$.
 
Now we recall that $Q$ is an eigenfunction of $H=H(0)$ with the discrete eigenvalue $E=-\lambda \in \sigma_{\mathrm{disc}}(H)$. By standard perturbation theory \cite{ReSi-78, Ka-95}, we find that $E(\kappa) \in \sigma_{\mathrm{disc}}(H(\kappa))$  provided that $|\kappa| \leq b$ with some sufficiently small number $b>0$.  Since the operators $H(\kappa) = H(\ii \, \mathrm{Im} \, \kappa)$ are unitarily equivalent, we see
$$
\sigma_{\mathrm{disc}}(H(\kappa)) = \sigma_{\mathrm{disc}}(H(\ii \, \mathrm{Im} \kappa)).
$$
Thus we deduce that $E \in \sigma_{\mathrm{disc}}(H(\kappa))$ for all $\kappa$ with $|\mathrm{Im} \, \kappa| < b$. Hence it follows from standard perturbation theory that the finite rank projections
$$
P(\kappa) = \frac{1}{2 \pi \ii} \oint_{|E-z|=r} (z-H(\kappa))^{-1} \, dz 
$$
with some small constant $r > 0$ are analytic in the strip $T_b = \{ \kappa \in \C^n : |\mathrm{Im} \, \kappa| < b \}$. We now apply Lemma \ref{lem:oconnor} to conclude that $U(\kappa) Q$ has an analytic continuation to the strip $T_b$, which shows that $e^{a |\cdot|} Q \in L^2(\R^n)$ for all $0 < a < b$.

The proof of Theorem \ref{thm:exp} is now complete. \hfill $\qed$

\begin{appendix}

\section{Auxiliary Results}

\begin{lemma} \label{lem:trick}
Suppose $P(D)$ satisfies Assumption \ref{ass:lin} with some $s>0$ and its multiplier $p(-\xi)=p(\xi)$ is an even function and let $\lambda \in \R$. Let $f \in H^s(\R^n)$ with $f:\R^n \to \C$ be of the form 
$$
f(x) = e^{\ii \vartheta} f_R(x)+ \ii e^{\ii \vartheta} f_I(x)
$$
with some constant $\vartheta \in \R$ and real-valued functions $f_R, f_I : \R^n \to \R$. Then we have
$$
\langle f, (P(D) +\lambda) f \rangle = \langle f_R, (P(D)+\lambda) f_R \rangle + \langle f_I, (P(D)+\lambda) f_I \rangle.
$$ 
Moreover, if $f \in L^q(\R^n)$ for some $2 < q < \infty$ then
$$
\| f \|_{L^q}^2 \leq \| f_R \|_{L^q}^2 + \| f_I \|_{L^q}^2,
$$
where equality holds if and only if $f_I = 0$ or $f_R^2 = \mu^2 f_I^2$ with some constant $\mu \geq 0$.
\end{lemma}

\begin{proof}
By subtracting the constant $\lambda$ from $p(\xi)$, we can assume without loss of generality that $\lambda = 0$ holds. Since $f_R, f_I : \R^n \to \R$ are real-valued, their Fourier transforms satisfy $\widehat{f}_R(-\xi) = \overline{\widehat{f}_R(\xi)}$ and $\widehat{f}_I(-\xi) = \overline{\widehat{f}_I(\xi)}$. Using that $p(-\xi)=p(\xi)$ is even and $|e^{\ii \vartheta} z|=|z|$ for all $z \in \C$, we calculate
\begin{align*}
\langle f,  P(D) f \rangle & = \int_{\R^n} p(\xi) |\widehat{f}_R(\xi) + \ii \widehat{f}_I(\xi)|^2 \,  d \xi = \int_{\R^n}  |\widehat{f}_{R}(\xi)|^2 \, d \xi + \int_{\R^n} p(\xi) |\widehat{f}_{I}(\xi)|^2 \, d\xi  \\
& \quad + \ii \int_{\R^n} p(\xi) \left [ \overline{\widehat{f}_R}(\xi) \widehat{f}_I(\xi)  - \widehat{f}_R(\xi) \overline{\widehat{f}_I}(\xi) \right ] \, d \xi  =  \langle f_R, P(D) f_R\rangle  + \langle f_I, P(D) f_I \rangle, 
\end{align*}
as claimed.

Assume now that $f \in L^q(\R^n)$ for some $2< q < \infty$. From the triangle inequality for the $L^{q/2}$-norm we find
$$
\| f \|_{L^q}^2 = \| |f_R|^2 + |f_I|^2 \|_{L^{q/2}} \leq \| |f_R|^2 \|_{L^{q/2}} + \| |f_I|^2 \|_{L^{q/2}} = \| f_R \|_{L^q}^2 + \|f_I \|_{L^q}^2.
$$
By the strict convexity of the $L^{q/2}$-norm for $2 < q <\infty$, we have equality if and only if $f_I=0$ or $f_R^2 = \mu^2 f_I^2$ for some constant $\mu \geq 0$.
\end{proof}

\begin{lemma}\label{lem:convo}
Let $f,g : \R^n \to [0, \infty)$ be two non-negative and continuous functions. Assume that their convolution
$$
(f \ast g)(x) = \int_{\R^n} f(x-y) g(y) \, d y 
$$
has finite values for all $x \in \R^n$. Then it holds that 
$$
\{ x \in \R^n : f \ast g > 0 \} = \{ x\in \R^n : f > 0 \} \oplus \{ x \in \R^n : g > 0 \} .
$$
where $A \oplus B =  \{ a + b : a \in A, b \in B \}$ denotes the Minkowski sum of two sets $A, B \subset \R^n$.
\end{lemma}

\begin{remark*}
{\em We could also allow that $(f \ast g)(x) = +\infty$ for some $x \in \R^n$ and the result remains valid. But since we apply this lemma iteratively in the proof of Theorem \ref{thm:main}, we assume that $(f \ast g)(x) < +\infty$ for all $x \in \R^n$.}
\end{remark*}

\begin{proof}
The proof is elementary. For the reader's convenience, we give the details.

Let us write $\Omega_f = \{ f > 0 \}$, $\Omega_g = \{ g > 0 \}$ and $\Omega_{f \ast g} = \{ f \ast g > 0 \}$. We suppose that both $f \not \equiv 0$ and $g \not \equiv 0$, since otherwise the claimed result trivially follows.

First, we show that $\Omega_f \oplus \Omega_g \subset \Omega_{f \ast g}$. Let $x = x_1 + x_2$ with $x_1 \in \Omega_f$ and $x_2 \in \Omega_g$. By continuity of $f$ and $g$, there exists some $\eps > 0$ such that $f > 0$ on $B_\eps(x_1)$ and $g > 0$ on $B_\eps(x_2)$. Thus, by using that $f \geq 0$ and $g \geq 0$ on all of $\R^n$, we get
$$
(f \ast g)(x)  = \int_{\R^n} f(x-y) g(y) \, d y \geq \int_{B_\eps(x_2)} f(x_1 + x_2 - y) g(y) \, d y  > 0,
$$
since $x_1 + x_2 - y \in B_\eps(x_1)$ when $y \in B_\eps(x_2)$. This shows that $\Omega_f \oplus \Omega_g \subset \Omega_{f \ast g}$.

Next, we prove that $\Omega_{f \ast g} \subset \Omega_f \oplus \Omega_g$ holds. Indeed, for every $x \in \R^n$, we can write
$$
(f \ast g)(x) = \int_{\R^n} f(x-y) g(y) \, d y = \int_{(\{x\} \ominus \Omega_f) \cap \Omega_g} f(x-y) g(y) \, d y,
$$ 
since $f(x-\cdot) \equiv 0$ on $\R^n \setminus (\{x \} \ominus \Omega_f)$ and $g \equiv 0$ on $\R^n \setminus \Omega_g$, where we denote $A \ominus B = \{ a-b : a \in A, \, b \in B\}$ for subsets $A$ and $B$ in $\R^n$. However, if $x \not \in \Omega_f \oplus \Omega_g$ then $(\{x \} \ominus \Omega_f) \cap \Omega_g = \emptyset$. Thus $(f \ast g)(x) = 0$ for any $x \not \in \Omega_f \oplus \Omega_g$, whence it follows that the inclusion $\Omega_{f \ast g} \subset \Omega_f \oplus \Omega_g$ is valid.
\end{proof}

\begin{lemma}[O'Connor's Lemma \cite{OCon-73}] \label{lem:oconnor}
Let $H$ be a Hilbert space and suppose $U(\kappa)$  are unitary operators on $H$ parametrized by $\kappa \in \R^n$. Let $P$ be a finite-rank projection on $H$ such that that $P(\kappa) = U(\kappa) P U(\kappa)^{-1}$ has an analytic continuation to $D=\{ z \in \C^n : |\mathrm{Im} \,z | < a \}$ for some $a > 0$. Then any $f \in \mathrm{ran} \, P$ has an analytic continuation from $D \cap \R$ to $D$ given by $f(\kappa)= U(\kappa) f$.
\end{lemma}
\end{appendix}

\end{document}